%% file: CategoryDynamicalSystem.tex
\documentclass[11pt]{article}

\usepackage{amsthm,amssymb, epsfig, epstopdf, tikz, tikz-cd, tcolorbox, subcaption, setspace, graphicx, hyperref, enumitem, mathrsfs, bbm, mathtools,color,stmaryrd, comment, fancyheadings, mathpazo, euler, subcaption,cleveref}
\theoremstyle{definition}

\usepackage[margin=.9in]{geometry}
\newtheorem{remark}{Remark}[section]

\newtheorem{theorem}{Theorem}
\newtheorem{definition}{Definition}[section]
\newtheorem{prop}{Proposition}[section]

\newtheorem{example}{Example}[section]

\begin{document}
\title{An Abstract Category  of Dynamical Systems} 
\author{James Schmidt\footnote{Department of Applied Mathematics and Statistics, \href{aschmi40@jhu.edu}{aschmi40@jhu.edu}}\\
\small{Johns Hopkins University Applied Physics Laboratory}\\
\small Johns Hopkins University}
\date{\today}
\maketitle
\input{macros}

\abstract{A notion of time is fundamental in the study of dynamical systems. Time arises as a standalone dynamical system and also in solutions or trajectories as a special kind of \textit{map} between systems.  We characterize time by a universal property and use universality to motivate an abstract definition for categories of dynamical systems. We propose this definition as guidance in concrete instantiations for other kinds of systems.}

\input{ctds-text}

\bibliographystyle{siam}
\bibliography{phdref}

\end{document}

%% file: macros.tex
\newcommand{\nat}[6][large]{%
  \begin{tikzcd}[ampersand replacement = \&, column sep=#1]
    #2\ar[bend left=40,""{name=U}]{r}{#4}\ar[bend right=40,',""{name=D}]{r}{#5}\& #3
          \ar[shorten <=10pt,shorten >=10pt,Rightarrow,from=U,to=D]{d}{~#6}
    \end{tikzcd}
}
\newcommand{\invamalg}{\mathbin{\text{\rotatebox[origin=c]{180}{$\amalg$}}}}

\newcommand{\dCrl}[0]{\mathfrak{dCrl}}
\newcommand{\ytil}{\tilde{y}}
\newcommand{\defeq}{\vcentcolon=}
\newcommand{\dee}{\partial}
\newcommand{\lb}{\{}
\newcommand{\rb}{\}}
\newcommand{\R}{\mathbb{R}}
\newcommand{\C}{\mathbb{C}}
\newcommand{\Q}{\mathbb{Q}}
\newcommand{\N}{\mathbb{N}}
\newcommand{\el}{\mathcal{L}}
\newcommand{\pdiv}[2]{\frac{\partial{#1}}{\partial{#2}}}
\newcommand{\discatp}{\displaystyle\bigsqcap}
\newcommand{\discats}{\displaystyle\bigsqcup}

\newcommand{\uZ}{\underline{\mathbb{Z}}}
\newcommand{\uF}[1]{\underline{\mathbb{F}}}
\newcommand{\one}{\mathbb{1}}
\newcommand{\two}{\mathbb{2}}

\newcommand{\dis}{\displaystyle}
\newcommand{\disp}{\displaystyle\prod}
\newcommand{\disu}{\displaystyle\bigcup}
\newcommand{\disi}{\displaystyle\bigcap}
\newcommand{\diss}{\displaystyle\sum}
\newcommand{\disg}{\displaystyle\int}
\newcommand{\disl}{\displaystyle\lim}
\newcommand{\dislim}{\displaystyle\lim}
\newcommand{\disliminf}{\displaystyle\liminf}
\newcommand{\dislimsup}{\displaystyle\limsup}
\newcommand{\disbop}{\displaystyle\bigotimes}
\newcommand{\disbos}{\displaystyle\bigoplus}
\newcommand{\dissup}{\displaystyle\sup}
\newcommand{\disinf}{\displaystyle\inf}
\newcommand{\dismax}{\displaystyle\max}
\newcommand{\dismin}{\displaystyle\min}
\newcommand{\dirlim}[1]{\displaystyle\varinjlim_{#1}}
\newcommand{\indlim}[1]{\displaystyle\varprojlim_{#1}}
\newcommand{\discatlim}{\indlim}
\newcommand{\discatcolim}{\dirlim}
\newcommand{\catcolim}{\mbox{colim}}
\newcommand{\catlim}{\mbox{lim}}

\newcommand{\colgray}[1]{\color{gray}{#1}\color{black}}

\newcommand{\sfM}{\sF{M}}

\newcommand{\forget}[2]{\Ub^{#1}_{#2}}

\newcommand\righttwoarrow{%
        \mathrel{\vcenter{\mathsurround0pt
                \ialign{##\crcr
                        \noalign{\nointerlineskip}$\rightarrow$\crcr
                        \noalign{\nointerlineskip}$\rightarrow$\crcr
                }%
        }}%
}

\newcommand{\Z}{\mathbb{Z}}
\newcommand{\Ab}[0]{\mathbb{A}}
\newcommand{\Bb}[0]{\mathbb{B}}
\newcommand{\Cb}[0]{\mathbb{C}}
\newcommand{\Db}[0]{\mathbb{D}}
\newcommand{\Eb}[0]{\mathbb{E}}
\newcommand{\Fb}[0]{\mathbb{F}}
\newcommand{\Gb}[0]{\mathbb{G}}
\newcommand{\Hb}[0]{\mathbb{H}}
\newcommand{\Ib}[0]{\mathbb{I}}
\newcommand{\Jb}[0]{\mathbb{J}}
\newcommand{\Kb}[0]{\mathbb{K}}
\newcommand{\Lb}[0]{\mathbb{L}}
\newcommand{\Mb}[0]{\mathbb{M}}
\newcommand{\Nb}[0]{\mathbb{N}}
\newcommand{\Ob}[0]{\mathbb{O}}
\newcommand{\Pb}[0]{\mathbb{P}}
\newcommand{\Qb}[0]{\mathbb{Q}}
\newcommand{\Rb}[0]{\mathbb{R}}
\newcommand{\Sb}[0]{\mathbb{S}}
\newcommand{\Tb}[0]{\mathbb{T}}
\newcommand{\Ub}[0]{\mathbb{U}}
\newcommand{\Vb}[0]{\mathbb{V}}
\newcommand{\Wb}[0]{\mathbb{W}}
\newcommand{\Xb}[0]{\mathbb{X}}
\newcommand{\Yb}[0]{\mathcal{Y}}
\newcommand{\Zb}[0]{\mathbb{Z}}

\newcommand{\sF}[1]{\mathsf{#1}}

\newcommand{\sC}[1]{\mathscr{#1}}

\newcommand{\mC}[1]{\mathcal{#1}}

\newcommand{\mB}[1]{\mathbb{#1}}

\newcommand{\mF}[1]{\mathfrak{#1}}

\newcommand{\Bc}[0]{\mathcal{B}}
\newcommand{\Cc}[0]{\mathcal{C}}
\newcommand{\Dc}[0]{\mathcal{D}}
\newcommand{\Ec}[0]{\mathcal{E}}
\newcommand{\Fc}[0]{\mathcal{F}}
\newcommand{\Gc}[0]{\mathcal{G}}
\newcommand{\Hc}[0]{\mathcal{H}}
\newcommand{\Ic}[0]{\mathcal{I}}
\newcommand{\Jc}[0]{\mathcal{J}}
\newcommand{\Kc}[0]{\mathcal{K}}
\newcommand{\Lc}[0]{\mathcal{L}}
\newcommand{\Mc}[0]{\mathcal{M}}
\newcommand{\Nc}[0]{\mathcal{N}}
\newcommand{\Oc}[0]{\mathcal{O}}
\newcommand{\Pc}[0]{\mathcal{P}}
\newcommand{\Qc}[0]{\mathcal{Q}}
\newcommand{\Rc}[0]{\mathcal{R}}
\newcommand{\Sc}[0]{\mathcal{S}}
\newcommand{\Tc}[0]{\mathcal{T}}
\newcommand{\Uc}[0]{\mathcal{U}}
\newcommand{\Vc}[0]{\mathcal{V}}
\newcommand{\Wc}[0]{\mathcal{W}}
\newcommand{\Xc}[0]{\mathcal{X}}
\newcommand{\Yc}[0]{\mathcal{Y}}
\newcommand{\Zc}[0]{\mathcal{Z}}

\newcommand{\aca}[0]{\mathcal{a}}
\newcommand{\bca}[0]{\mathcal{b}}
\newcommand{\cca}[0]{\mathcal{c}}
\newcommand{\dca}[0]{\mathcal{d}}
\newcommand{\eca}[0]{\mathcal{e}}
\newcommand{\fca}[0]{\mathcal{f}}
\newcommand{\gca}[0]{\mathcal{g}}
\newcommand{\hca}[0]{\mathcal{h}}
\newcommand{\ica}[0]{\mathcal{i}}
\newcommand{\jca}[0]{\mathcal{j}}
\newcommand{\kca}[0]{\mathcal{k}}
\newcommand{\lca}[0]{\mathcal{l}}
\newcommand{\mca}[0]{\mathcal{m}}
\newcommand{\nca}[0]{\mathcal{n}}
\newcommand{\oca}[0]{\mathcal{o}}
\newcommand{\pca}[0]{\mathcal{p}}
\newcommand{\qca}[0]{\mathcal{q}}
\newcommand{\rca}[0]{\mathcal{r}}
\newcommand{\sca}[0]{\mathcal{s}}
\newcommand{\tca}[0]{\mathcal{t}}
\newcommand{\uca}[0]{\mathcal{u}}
\newcommand{\vca}[0]{\mathcal{v}}
\newcommand{\wca}[0]{\mathcal{w}}
\newcommand{\xca}[0]{\mathcal{x}}
\newcommand{\yca}[0]{\mathcal{y}}
\newcommand{\zca}[0]{\mathcal{z}}

\newcommand{\Af}[0]{\mathfrak{A}}
\newcommand{\Bf}[0]{\mathfrak{B}}
\newcommand{\Cf}[0]{\mathfrak{C}}
\newcommand{\Df}[0]{\mathfrak{D}}
\newcommand{\Ef}[0]{\mathfrak{E}}
\newcommand{\Ff}[0]{\mathfrak{F}}
\newcommand{\Gf}[0]{\mathfrak{G}}
\newcommand{\Hf}[0]{\mathfrak{H}}
\newcommand{\If}[0]{\mathfrak{I}}
\newcommand{\Jf}[0]{\mathfrak{J}}
\newcommand{\Kf}[0]{\mathfrak{K}}
\newcommand{\Lf}[0]{\mathfrak{L}}
\newcommand{\Mf}[0]{\mathfrak{M}}
\newcommand{\Nf}[0]{\mathfrak{N}}
\newcommand{\Of}[0]{\mathfrak{O}}
\newcommand{\Pf}[0]{\mathfrak{P}}
\newcommand{\Qf}[0]{\mathfrak{Q}}
\newcommand{\Rf}[0]{\mathfrak{R}}
\newcommand{\Sf}[0]{\mathfrak{S}}
\newcommand{\Tf}[0]{\mathfrak{T}}
\newcommand{\Uf}[0]{\mathfrak{U}}
\newcommand{\Vf}[0]{\mathfrak{V}}
\newcommand{\Wf}[0]{\mathfrak{W}}
\newcommand{\Xf}[0]{\mathfrak{X}}
\newcommand{\Yf}[0]{\mathfrak{Y}}
\newcommand{\Zf}[0]{\mathfrak{Z}}

\newcommand{\af}[0]{\mathfrak{a}}
\newcommand{\bff}[0]{\mathfrak{b}}
\newcommand{\cf}[0]{\mathfrak{c}}
\newcommand{\dff}[0]{\mathfrak{d}}
\newcommand{\ef}[0]{\mathfrak{e}}
\newcommand{\ff}[0]{\mathfrak{f}}
\newcommand{\gf}[0]{\mathfrak{g}}
\newcommand{\hf}[0]{\mathfrak{h}}
\newcommand{\ifrak}{\mathfrak{i}}
\newcommand{\jf}[0]{\mathfrak{j}}
\newcommand{\kf}[0]{\mathfrak{k}}
\newcommand{\lf}[0]{\mathfrak{l}}
\newcommand{\mf}[0]{\mathfrak{m}}
\newcommand{\nf}[0]{\mathfrak{n}}
\newcommand{\of}[0]{\mathfrak{o}}
\newcommand{\pf}[0]{\mathfrak{p}}
\newcommand{\qf}[0]{\mathfrak{q}}
\newcommand{\rf}[0]{\mathfrak{r}}
\renewcommand{\sf}[0]{\mathfrak{s}}
\newcommand{\tf}[0]{\mathfrak{t}}
\newcommand{\uf}[0]{\mathfrak{u}}
\newcommand{\vf}[0]{\mathfrak{v}}
\newcommand{\wf}[0]{\mathfrak{w}}
\newcommand{\xf}[0]{\mathfrak{x}}
\newcommand{\yf}[0]{\mathfrak{y}}
\newcommand{\zf}[0]{\mathfrak{z}}
\newcommand{\cdX}{\mathfrak{cdX}}
\newcommand{\cdCrl}{\mathfrak{cdCrl}}

\newcommand{\scA}[0]{\mathscr{A}}
\newcommand{\scB}[0]{\mathscr{B}}
\newcommand{\scC}[0]{\mathscr{C}}
\newcommand{\scD}[0]{\mathscr{D}}
\newcommand{\scE}[0]{\mathscr{E}}
\newcommand{\scF}[0]{\mathscr{F}}
\newcommand{\scG}[0]{\mathscr{G}}
\newcommand{\scH}[0]{\mathscr{H}}
\newcommand{\scI}[0]{\mathscr{I}}
\newcommand{\scJ}[0]{\mathscr{J}}
\newcommand{\scK}[0]{\mathscr{K}}
\newcommand{\scL}[0]{\mathscr{L}}
\newcommand{\scM}[0]{\mathscr{M}}
\newcommand{\scN}[0]{\mathscr{N}}
\newcommand{\scO}[0]{\mathscr{O}}
\newcommand{\scP}[0]{\mathscr{P}}
\newcommand{\scQ}[0]{\mathscr{Q}}
\newcommand{\scR}[0]{\mathscr{R}}
\newcommand{\scS}[0]{\mathscr{S}}
\newcommand{\scT}[0]{\mathscr{T}}
\newcommand{\scU}[0]{\mathscr{U}}
\newcommand{\scV}[0]{\mathscr{V}}
\newcommand{\scW}[0]{\mathscr{W}}
\newcommand{\scX}[0]{\mathscr{X}}
\newcommand{\scY}[0]{\mathscr{Y}}
\newcommand{\scZ}[0]{\mathscr{Z}}

\newcommand{\fA}[0]{\mathsf{A}}
\newcommand{\fB}[0]{\mathsf{B}}
\newcommand{\fC}[0]{\mathsf{C}}
\newcommand{\fD}[0]{\mathsf{D}}
\newcommand{\fE}[0]{\mathsf{E}}
\newcommand{\fG}[0]{\mathsf{G}}
\newcommand{\fH}[0]{\mathsf{H}}
\newcommand{\fI}[0]{\mathsf{I}}
\newcommand{\fJ}[0]{\mathsf{J}}
\newcommand{\fK}[0]{\mathsf{K}}
\newcommand{\fL}[0]{\mathsf{L}}
\newcommand{\fM}[0]{\mathsf{M}}
\newcommand{\fN}[0]{\mathsf{N}}
\newcommand{\fO}[0]{\mathsf{O}}
\newcommand{\fP}[0]{\mathsf{P}}
\newcommand{\fQ}[0]{\mathsf{Q}}
\newcommand{\fR}[0]{\mathsf{R}}
\newcommand{\fS}[0]{\mathsf{S}}
\newcommand{\fT}[0]{\mathsf{T}}
\newcommand{\fU}[0]{\mathsf{U}}
\newcommand{\fV}[0]{\mathsf{V}}
\newcommand{\fW}[0]{\mathsf{W}}
\newcommand{\fX}[0]{\mathsf{X}}
\newcommand{\fY}[0]{\mathsf{Y}}
\newcommand{\fZ}[0]{\mathsf{Z}}

\newcommand{\fa }[0]{\mathsf{a}}
\newcommand{\fb }[0]{\mathsf{b}}
\newcommand{\fc }[0]{\mathsf{c}}
\newcommand{\fd}[0]{\mathsf{d}}
\newcommand{\fe}[0]{\mathsf{e}}
\newcommand{\fg}[0]{\mathsf{g}}
\newcommand{\fh}[0]{\mathsf{h}}
\newcommand{\fj}[0]{\mathsf{j}}
\newcommand{\fk}[0]{\mathsf{k}}
\newcommand{\fl }[0]{\mathsf{l}}
\newcommand{\fm }[0]{\mathsf{m}}
\newcommand{\fn }[0]{\mathsf{n}}
\newcommand{\fo }[0]{\mathsf{o}}
\newcommand{\fp}[0]{\mathsf{p}}
\newcommand{\fq}[0]{\mathsf{q}}
\newcommand{\fr}[0]{\mathsf{r}}
\newcommand{\fs}[0]{\mathsf{s}}
\newcommand{\ft }[0]{\mathsf{t}}
\newcommand{\fu }[0]{\mathsf{u}}
\newcommand{\fv }[0]{\mathsf{v}}
\newcommand{\fw}[0]{\mathsf{w}}
\newcommand{\fx}[0]{\mathsf{x}}
\newcommand{\fy}[0]{\mathsf{y}}
\newcommand{\fz}[0]{\mathsf{z}}

\newcommand{\dA}[0]{\dot{A}}
\newcommand{\dB}[0]{\dot{B}}
\newcommand{\dC}[0]{\dot{C}}
\newcommand{\dD}[0]{\dot{D}}
\newcommand{\dE}[0]{\dot{E}}
\newcommand{\dF}[0]{\dot{F}}
\newcommand{\dG}[0]{\dot{G}}
\newcommand{\dH}[0]{\dot{H}}
\newcommand{\dI}[0]{\dot{I}}
\newcommand{\dJ}[0]{\dot{J}}
\newcommand{\dK}[0]{\dot{K}}
\newcommand{\dL}[0]{\dot{L}}
\newcommand{\dM}[0]{\dot{M}}
\newcommand{\dN}[0]{\dot{N}}
\newcommand{\dO}[0]{\dot{O}}
\newcommand{\dP}[0]{\dot{P}}
\newcommand{\dQ}[0]{\dot{Q}}
\newcommand{\dR}[0]{\dot{R}}
\newcommand{\dS}[0]{\dot{S}}
\newcommand{\dT}[0]{\dot{T}}
\newcommand{\dU}[0]{\dot{U}}
\newcommand{\dV}[0]{\dot{V}}
\newcommand{\dW}[0]{\dot{W}}
\newcommand{\dX}[0]{\dot{X}}
\newcommand{\dY}[0]{\dot{Y}}
\newcommand{\dZ}[0]{\dot{Z}}

\newcommand{\da}[0]{\dot{a}}
\newcommand{\db}[0]{\dot{b}}
\newcommand{\dc}[0]{\dot{c}}
\newcommand{\dd}[0]{\dot{d}}
\newcommand{\de}[0]{\dot{e}}
\newcommand{\df}[0]{\dot{f}}
\newcommand{\dg}[0]{\dot{g}}
\renewcommand{\dh}[0]{\dot{h}}
\newcommand{\di}[0]{\dot{i}}
\renewcommand{\dj}[0]{\dot{j}}
\newcommand{\dk}[0]{\dot{k}}
\newcommand{\dl}[0]{\dot{l}}
\newcommand{\dm}[0]{\dot{m}}
\newcommand{\dn}[0]{\dot{n}}
\newcommand{\dq}[0]{\dot{q}}
\newcommand{\dr}[0]{\dot{r}}
\newcommand{\ds}[0]{\dot{s}}
\newcommand{\dt}[0]{\dot{t}}
\newcommand{\du}[0]{\dot{u}}
\newcommand{\dv}[0]{\dot{v}}
\newcommand{\dw}[0]{\dot{w}}
\newcommand{\dx}[0]{\dot{x}}
\newcommand{\dy}[0]{\dot{y}}
\newcommand{\dz}[0]{\dot{z}}

\newcommand{\oA}[0]{\overline{A}}
\newcommand{\oB}[0]{\overline{B}}
\newcommand{\oC}[0]{\overline{C}}
\newcommand{\oD}[0]{\overline{D}}
\newcommand{\oE}[0]{\overline{E}}
\newcommand{\oF}[0]{\overline{F}}
\newcommand{\oG}[0]{\overline{G}}
\newcommand{\oH}[0]{\overline{H}}
\newcommand{\oI}[0]{\overline{I}}
\newcommand{\oJ}[0]{\overline{J}}
\newcommand{\oK}[0]{\overline{K}}
\newcommand{\oL}[0]{\overline{L}}
\newcommand{\oM}[0]{\overline{M}}
\newcommand{\oN}[0]{\overline{N}}
\newcommand{\oO}[0]{\overline{O}}
\newcommand{\oP}[0]{\overline{P}}
\newcommand{\oQ}[0]{\overline{Q}}
\newcommand{\oR}[0]{\overline{R}}
\newcommand{\oS}[0]{\overline{S}}
\newcommand{\oT}[0]{\overline{T}}
\newcommand{\oU}[0]{\overline{U}}
\newcommand{\oV}[0]{\overline{V}}
\newcommand{\oW}[0]{\overline{W}}
\newcommand{\oX}[0]{\overline{X}}
\newcommand{\oY}[0]{\overline{Y}}
\newcommand{\oZ}[0]{\overline{Z}}

\newcommand{\oa}[0]{\overline{a}}
\newcommand{\ob}[0]{\overline{b}}
\newcommand{\oc}[0]{\overline{c}}
\newcommand{\od}[0]{\overline{d}}
\renewcommand{\oe}[0]{\overline{e}}
\newcommand{\og}[0]{\overline{g}}
\newcommand{\oh}[0]{\overline{h}}
\newcommand{\oi}[0]{\overline{i}}
\newcommand{\oj}[0]{\overline{j}}
\newcommand{\ok}[0]{\overline{k}}
\newcommand{\ol}[0]{\overline{l}}
\newcommand{\om}[0]{\overline{m}}
\newcommand{\on}[0]{\overline{n}}
\newcommand{\oo}[0]{\overline{o}}
\newcommand{\op}[0]{\overline{p}}
\newcommand{\oq}[0]{\overline{q}}
\newcommand{\os}[0]{\overline{s}}
\newcommand{\ot}[0]{\overline{t}}
\newcommand{\ou}[0]{\overline{u}}
\newcommand{\ov}[0]{\overline{v}}
\newcommand{\ow}[0]{\overline{w}}
\newcommand{\ox}[0]{\overline{x}}
\newcommand{\oy}[0]{\overline{y}}
\newcommand{\oz}[0]{\overline{z}}

\renewcommand{\a}{\alpha}
\renewcommand{\b}{\beta}
\renewcommand{\d}{\delta}
\newcommand{\e}{\varepsilon}
\newcommand{\f}{\phi}
\newcommand{\g}{\gamma}
\newcommand{\h}{\eta}
\renewcommand{\i}{\iota}
\renewcommand{\k}{\kappa}
\renewcommand{\l}{\lambda}
\newcommand{\m}{\mu}
\newcommand{\n}{\nu}
\newcommand{\p}{\pi}
\newcommand{\ph}{\varphi}
\newcommand{\ps}{\psi}
\newcommand{\q}{\xi}
\renewcommand{\r}{\rho}
\newcommand{\s}{\sigma}
\renewcommand{\t}{\tau}
\renewcommand{\v}{\upsilon}
\newcommand{\x}{\chi}
\newcommand{\z}{\zeta}
\newcommand{\G}{\Gamma}

\newcommand{\aarb}[0]{\<a>}
\newcommand{\barb}[0]{\<b>}
\newcommand{\carb}[0]{\<c>}
\newcommand{\darb}[0]{\<d>}
\newcommand{\earb}[0]{\<e>}
\newcommand{\farb}[0]{\<f>}
\newcommand{\garb}[0]{\<g>}
\newcommand{\harb}[0]{\<h>}
\newcommand{\iarb}[0]{\<i>}
\newcommand{\jarb}[0]{\<j>}
\newcommand{\karb}[0]{\<k>}
\newcommand{\larb}[0]{\<l>}
\newcommand{\marb}[0]{\<m>}
\newcommand{\narb}[0]{\<n>}
\newcommand{\oarb}[0]{\<o>}
\newcommand{\parb}[0]{\<p>}
\newcommand{\qarb}[0]{\<q>}
\newcommand{\rarb}[0]{\<r>}
\newcommand{\sarb}[0]{\<s>}
\newcommand{\tarb}[0]{\<t>}
\newcommand{\uarb}[0]{\<u>}
\newcommand{\varb}[0]{\<v>}
\newcommand{\warb}[0]{\<w>}
\newcommand{\xarb}[0]{\<x>}
\newcommand{\yarb}[0]{\<y>}
\newcommand{\zarb}[0]{\<z>}

\newcommand{\hA}[0]{\hat{A}}
\newcommand{\hB}[0]{\hat{B}}
\newcommand{\hC}[0]{\hat{C}}
\newcommand{\hD}[0]{\hat{D}}
\newcommand{\hE}[0]{\hat{E}}
\newcommand{\hF}[0]{\hat{F}}
\newcommand{\hG}[0]{\hat{G}}
\newcommand{\hH}[0]{\hat{H}}
\newcommand{\hI}[0]{\hat{I}}
\newcommand{\hJ}[0]{\hat{J}}
\newcommand{\hK}[0]{\hat{K}}
\newcommand{\hL}[0]{\hat{L}}
\newcommand{\hM}[0]{\hat{M}}
\newcommand{\hN}[0]{\hat{N}}
\newcommand{\hO}[0]{\hat{O}}
\newcommand{\hP}[0]{\hat{P}}
\newcommand{\hQ}[0]{\hat{Q}}
\newcommand{\hR}[0]{\hat{R}}
\newcommand{\hS}[0]{\hat{S}}
\newcommand{\hT}[0]{\hat{T}}
\newcommand{\hU}[0]{\hat{U}}
\newcommand{\hV}[0]{\hat{V}}
\newcommand{\hW}[0]{\hat{W}}
\newcommand{\hX}[0]{\hat{X}}
\newcommand{\hY}[0]{\hat{Y}}
\newcommand{\hZ}[0]{\hat{Z}}

\newcommand{\ha}[0]{\hat{a}}
\newcommand{\hb}[0]{\hat{b}}
\newcommand{\hc}[0]{\hat{c}}
\newcommand{\hd}[0]{\hat{d}}
\newcommand{\he}[0]{\hat{e}}
\newcommand{\hg}[0]{\hat{g}}
\newcommand{\hh}[0]{\hat{h}}
\newcommand{\hi}[0]{\hat{i}}
\newcommand{\hj}[0]{\hat{j}}
\newcommand{\hk}[0]{\hat{k}}
\newcommand{\hl}[0]{\hat{l}}
\newcommand{\hm}[0]{\hat{m}}
\newcommand{\hn}[0]{\hat{n}}
\newcommand{\ho}[0]{\hat{o}}
\newcommand{\hp}[0]{\hat{p}}
\newcommand{\hq}[0]{\hat{q}}
\newcommand{\hr}[0]{\hat{r}}
\newcommand{\hs}[0]{\hat{s}}
\newcommand{\hu}[0]{\hat{u}}
\newcommand{\hv}[0]{\hat{v}}
\newcommand{\hw}[0]{\hat{w}}
\newcommand{\hx}[0]{\hat{x}}
\newcommand{\hy}[0]{\hat{y}}
\newcommand{\hz}[0]{\hat{z}}

\newcommand{\hyph}[2]{\Fc_{#1}:#2\rightarrow \sF{RelMan^c}}
\newcommand{\repsys}{(\Fc_{N^\scT}:\N\rightarrow\sF{RelMan^c},\frac{d}{dt})}
\newcommand{\truerep}[1]{\left(\Fc_{N^\scT}{#1}:\N\rightarrow\sF{RelMan^c},\frac{D}{Dt}\right)}

%% file: ctds-text.tex
\section{Introduction}

Dynamical systems exhibit behavior and are often described in terms of the laws which govern behavior. In continuous-time systems, ordinary differential equations (odes) represent such laws and their solutions realize them as behaviors, while in discrete-time systems endomaps both specify the law of which iterates are its realization. Emphasis on behavior is due to Willems \cite{willems}.  Common to various, if not all, notions of dynamical systems is the presence of time according to which behaviors evolve. We suggest that time is a defining feature of dynamical systems and characterize time abstractly by a universal property, which we then use to produce a general definition of abstract dynamical system. 

Because time appears both as a system and in maps, we will flow freely between  behavioral and behavior-governing perspectives. In fact, viewing time through its universal property unifies these perspectives. To preview how the construction works, consider an ode in $\R$ defined by $\dot{t} = 1$, whose solutions exhibit a translate of the flow of time with initial condition corresponding to translation of initial time. Bracketing completeness issues, the fact that  odes have solutions expresses that there is a map of odes from the one which defines (the flow of) time. Specification of initial condition, moreover, uniquely determines the map. In this manner, we demarcate  systems representing time as \textit{universal}, though the construction in particular cases may require some care.

We highlight additional benefits of considering systems category theoretically. In particular, the notion of map (alternatively: morphism) plays a particularly central role in a study of dynamical systems. Conveniently, many systems-theoretic properties arise \textit{as} maps of systems, namely maps between state spaces which respect the dynamics. We have already identified one, solutions of odes as a map from time.   But other  properties are also describable as maps, such as equilibria points (in continuous time) or fixed points (in discrete). Recognizing such properties as maps has the added advantage of guaranteeing their preservation under maps of dynamical systems: in a category, morphisms compose, which implies, in particular, that maps preserve solutions and equilibria points.   This observation led to a category-theoretic construction of hybrid system in \cite{lermanhybrid} and was extended in \cite{schmidt2019morphisms} and  \cite{lermanSchmidt1}.

This paper is organized as follows: in \cref{sec:dynamicalSystems}, we review continuous-time and discrete-time dynamical systems, and explain in each case how time appears as a kind of system and is universal. For discrete-time systems (\cref{subsec:timeDiscrete}) and complete continuous-time systems (\cref{subsec:timeCDS}) much of this task amounts to aligning the relevant categorical concepts. The case of arbitrary continuous-time systems (\cref{subsec:timeDS}) is more subtle. That in each case we can realize universality of time provides confidence that the category theory tracks the systems theory. In \cref{sec:abstractSystem}, we then construct an abstract notion of system, of which those in \cref{sec:dynamicalSystems} and others in \cite{schmidt2019morphisms} are instances. 

 
\section{Ordinary Dynamical Systems}\label{sec:dynamicalSystems}
We start with review of familiar concepts from the study of continuous-time dynamical systems, and consider their categorical interpretation in \cref{subsec:review}. We then proceed in \cref{subsec:timeCDS} to explaining how $\big(\R,\dot{t}=1\big)$ is a universal complete dynamical system. Most continuous-time systems are not complete, yet the sense in which $(\R,\dt = 1)$ is universal remains, with modification, unperturbed. We work through the requisite modification in \cref{subsec:timeDS}. We conclude in \ref{subsec:timeDiscrete} with an example from \cite{riehl} for discrete-time systems. 
\subsection{Continuous-Time Systems}\label{subsec:review}
The theory of continuous-time systems is a theory of ordinary differential equations. The theory generalizes to manifolds.  
\begin{definition}\label{def:c0-dysys}
We define a continuous-time dynamical system to be a pair $(M,X)$ where $M$ is a smooth manifold and $X\in \Xf(M)$ a smooth vector field on $M$. 
\end{definition}

Recall that a vector field $X\in \Xf(M)$ is a smooth section of the tangent bundle $\begin{tikzcd}
	M\arrow[r,shift left,"X"] & \arrow[l,shift left,"p_M"] TM,
\end{tikzcd}$ satisfying $p_M\circ X = id_M$ for canonical projection $p_M:TM\rightarrow M$.  The section is smooth if the map $X:M\rightarrow TM$ is smooth as a map of manifolds.

At the outset, we take all smooth manifolds to be Hausdorff and second countable. One may consider as a running example a Euclidean state space   $M=\R^n$ and an ordinary differential equation $\dx = X(x)$ defined by the vector field.

We present a notion of map \textit{between} two dynamical systems. To do so, we  relate how  maps  behave on tangent vectors in the source. A map $f:M\rightarrow N$ between manifolds induces a collection of maps, pushforwards, between tangent spaces $\left\{Tf_p:T_pM\rightarrow T_{f(p)}N\right\}_{p\in M}$ defined by $Tf_pv(\eta)\defeq v(\eta\circ f)$ for $v\in T_pM$ and $\eta\in \Cc^\infty(N)$. In general, this collection of maps of tangent vectors does not induce a map of vector fields because we are not guaranteed  $Tf_pv = Tf_{p'}v' \in T_{f(p)}N$ for two arbitrary tangent vectors $v\in T_pM,\, v'$ in $T_{p'}M$ where $f(p) = f(p')$, and therefore that $X(p)=X(p')$ of vector field $X\in \Xf(M)$. Apriori, therefore, there is no reason why the tangent vector defined at $f(x)$ by $Y$ should agree with the pushforward under $Tf$ of $X(x)$. We isolate pairs of vector fields which do cohere with the map: 

\begin{definition}\label{def:relatedVectorFields} Let $f:M\rightarrow N$ be a smooth map of manifolds.  We say that vector fields $\big(X\in \Xf(M),Y\in \Xf(N)\big)$ are $f$-\textit{related} if $Tf\circ X = Y\circ f$.	
\end{definition}

Relatedness is the condition we need for  a notion of map of dynamical systems:

\begin{definition}\label{def:c1-dysys}
	Let $(M,X)$ and $(N,Y)$ be two continuous-time dynamical systems.  A  \textit{map} (or \textit{morphism}) $(M,X)\xrightarrow{f}(N,Y)$ of systems is a smooth map $M\xrightarrow{f}N$ of manifolds such that $(X,Y)$ are $f$-related (\cref{def:relatedVectorFields}). 
\end{definition}

Now we may articulate how common systems phenomena arise as maps of dynamical systems. 
\begin{example}\label{ex:dysys1}
	An equlibrium point $x_{e}\in M$ for system $(M,X)$ is a point whose dynamics are zero: $X(x_e) = 0$. When such a point exists, there is a map from one point system  $\iota:(\star,0)\hookrightarrow (M,X)$ sending $\star\mapsto x_e$. Relatedness of vector fields requires that $X(\iota(\star))=0$, whenever there is such a map.  
\end{example}

\begin{example}\label{ex:dysys2}
	Consider the interval $A=[0,1]/\sim$ with endpoints identified $0\sim 1$ and dynamics given by constant vector field $\frac{d}{dt}$. Then $A\simeq S^1$ and a map $\left(A,\frac{d}{dt}\right)\xrightarrow{\eta}(M,X)$ defines a periodic orbit. 
\end{example}

\begin{example}\label{ex:dysys3}
	Trajectories of a system $(M,X)$ arise as maps  $\gamma:\left(\R,\frac{d}{dt}\right)\hookrightarrow(M,X)$. Relatedness of vector fields translates that $\g$ is a solution of the ode $X(\g(t))=\dot{\gamma}(t)$ (\cref{eq:relatedMapRM}, \cref{def:solutionToDynamicalSystem}).
\end{example}

This last example isolates a special class of dynamical systems. Recall the definition of solution: 
\begin{definition}\label{def:solutionToDynamicalSystem}
Let $(M,X)$ be a continuous-time dynamical system.  A  \textit{solution}, or \textit{integral curve},  $\ph_{X}$ of $(M,X)$ is a map $\ph_{X}:(t_0,t_1)\rightarrow M$, for some $t_0<0<t_1$, such that $\frac{d}{dt}\ph_{X}(t) = X(\ph_{X}(t))$ for all $t\in (t_0,t_1)$. The value $\ph_X(0)=x_0$ at $t=0$ is called the \textit{initial condition}, and we may write $\ph_{X,x_0}$ to indicate  that $\ph_X$ has initial condition $x_0$. 

 We say that $\ph_X$ is \textit{maximal} if its domain may not be extended, i.e.\ if there is no $(t_0',t_1')\supsetneq (t_0,t_1)$ for which $\psi_X:(t_0',t_1')\rightarrow M$ is an integral curve with initial condition $\psi_X(0) = x_0$. When the domain of maximal integral curve $\ph_{X,x_0}$ is $\R$, we say that $\ph_{X,x_0}$ is \textit{complete} and that $(M,X)$ is complete when $\ph_{X,x_0}$ is complete for all $x_0\in M$.  \end{definition}

	 Every dynamical system $(M,X)$ has  solutions.  Moreover, solutions with specified initial condition are  unique. We recall and restate the central Existence and Uniqueness Theorem: 
	 
\begin{theorem}\label{theorem:E&U}
	Let $(M,X)$ be a dynamical system, and $x_0\in M$. Then there are $t_0<0<t_1\in \R$ for which smooth map $\ph_{X,x_0}:(t_0,t_1)\rightarrow M$ is unique maximal solution of $(M,X)$ with initial condition $x_0$.  Thus,  $\ph_{X,x_0}(0) = 0$ and $\frac{d}{dt} \ph_{X,x_0}(t) = X(\ph_{X,x_0}(t))$ for $t\in (t_0,t_1)$.  Moreover, given curve $\g:(t_0',t_1')\rightarrow M$ satisfying $\g(0)= x_0$ and $\frac{d}{dt}\g(t) = X(\g(t))$, then $(t_0',t_1')\subseteq (t_0,t_1)$ and $\g(t) = \ph_X(t)$ for $t\in (t_0',t_1')$. 
\end{theorem}
\begin{proof}
	See \cite[\S 14.3]{tu}.
\end{proof}

We now collect the observation in \cref{ex:dysys3} into an equivalent definition of \cref{def:solutionToDynamicalSystem}. 
\begin{definition}\label{def:integralCurveAsMap}
	Let $(M,X)$ be a continuous-time dynamical system.  A \textit{solution}  (or \textit{integral curve}) \textit{of system} $(M,X)$  is a map $\ph_{X,x_0}:\left((t_0,t_1),\frac{d}{dt}\right)\rightarrow (M,X)$ of dynamical systems from the dynamical system $\left((t_0,t_1),\frac{d}{dt}\right)$ with constant vector field $\frac{d}{dt}\in \Xf(\R)$ sending $t\mapsto 1\in T_t\R$. 
\end{definition}

Equivalence of \cref{def:solutionToDynamicalSystem} and \cref{def:integralCurveAsMap} follows from \cref{def:c1-dysys}, since \begin{equation}\label{eq:relatedMapRM} X\circ \ph_{X,x_0} = T\ph_{X,x_0} \left(\frac{d}{dt}\right) = \frac{d}{dt} \ph_{X,x_0}.\end{equation} 

A convenient consequence of \cref{def:integralCurveAsMap} is that maps $(M,X)\xrightarrow{f}(N,Y)$ of dynamical systems preserve integral curves: the composition of maps \begin{equation}\label{eq:mapsCompose}\begin{tikzcd}\left((t_0,t_1),\frac{d}{dt}\right)\arrow[r,"\ph_{X,x0}"]\arrow[rr,bend left] &(M,X)\arrow[r,"f"]&(N,Y)\end{tikzcd}\end{equation} is indeed a map of dynamical systems. Compositionality of maps follows from the chain rule: for $(M,X)\xrightarrow{f}(N,Y)\xrightarrow{g}(P,Z)$, $T(g\circ f) = Tg\circ Tf$. 

With preliminary review complete, we turn to interpreting these notions category theoretically.

\subsection{Complete Time is a Universal Complete Dynamical System}\label{subsec:timeCDS}

	Dynamical systems and maps assemble to define a category of dynamical systems. 
\begin{definition}\label{def:c-dysys}
The category $\sF{DySys}$ of \textit{continuous-time dynamical systems} has dynamical systems as objects (\cref{def:c0-dysys}) and maps of dynamical systems as morphisms (\cref{def:c1-dysys}). 
\end{definition}

Complete dynamical systems (\cref{def:solutionToDynamicalSystem}) and their morphisms also form a category, which we here denote by $\sF{comDySys}$.  It is a full subcategory of the category of dynamical systems (\cref{def:c-dysys}) whose objects may not have complete integral  curves. 

In $\sF{comDySys}$, consider that the collection of solution maps $\big\{\ph_{X,x_0}:\R\rightarrow M\big\}_{x_0\in M}$ also defines map $$\ph_{X,(\cdot)}:M\rightarrow M^\R$$ sending $x_0\mapsto \ph_{X,x_0}$. And the collection of maps $\big\{\ph_{X,(\cdot)}:M\rightarrow M^\R\big\}_{(M,X)\in \sF{comDySys}}$ itself arises as a map from objects in $\sF{comDySys}$ to $\sF{Set}$. It turns out, in fact, that  there are functors  with respect to which the maps $\ph$, ranging over systems $(M,X)$,  assemble  into a \textit{natural transformation}. 

Consider forgetful functor $\upsilon:\sF{comDySys}\rightarrow\sF{Set}$ sending $(M,X)\mapsto \{x\in M\}$ which drops dynamics $X$ and smooth/topological structure of manifold $M$. Next, let $$\sF{comDySys}\left(\left(\R,\frac{d}{dt}\right),\bullet\right):\sF{comDySys}\rightarrow\sF{Set}$$ be the representing functor, represented by $\left(\R,\frac{d}{dt}\right)$,  sending $(M,X)\mapsto \sF{comDySys}\left(\big(\R,\frac{d}{dt}\big),(M,X)\right)$, the collection of integral curve of $(M,X)$.  Then we have commuting diagram 

\begin{equation}\label{eq:naturalityOfExp}
\begin{tikzcd}	
\big\{x\in M\big\}\arrow[rr,"\ph_{X,(\cdot)}"]\arrow[dd,"\upsilon f"] & & \sF{comDySys}\left(\left(\R,\frac{d}{dt}\right),(M,X)\right)\arrow[dd,"f_*"] \\ 
\\ 
\big\{y\in N\}\arrow[rr,"\ph_{Y,(\cdot)}"] & &  \sF{comDySys}\left(\left(\R,\frac{d}{dt}\right),(N,Y)\right).
\end{tikzcd}
\end{equation}
By \eqref{eq:mapsCompose} and \cref{def:integralCurveAsMap}, $f_*\circ \ph_{X,x_0}:\left(\R,\frac{d}{dt}\right)\rightarrow(N,Y)$ is an integral curve with $f_*\big(\ph_{X,x_0}(0)\big) = f(x_0)$ which by \cref{theorem:E&U}, must be $\ph_{Y,f(x_0)}$. In fact, \cref{theorem:E&U} implies more, namely that the natural transformation $\ph:\upsilon\Rightarrow \sF{comDySys}\left(\left(\R,\frac{d}{dt}\right),\cdot\right)$ is a bijection and therefore a natural \textit{isomorphism}.  By Yoneda, then, $\upsilon$ is representable.

In $\sF{comDySys}$, existence and uniqueness can thus be formulated in Yoneda categorical dress: 

\begin{prop}\label{prop:e&uRepresentable} 
	The forgetful functor $\upsilon:\sF{comDySys}\rightarrow\sF{Set}$---sending continuous-time dynamical system $(M,X)\mapsto \{x\in M\}$ to the underlying set---is representable. 
\end{prop}
\begin{proof}
	We argue that the element $\left(\big(\R,\frac{d}{dt}\big),0\right) \in \disg_{\sF{DySys}}\upsilon$ is initial in the category of elements \cite[Proposition 2.4.8]{riehl}. By assumption, given dynamical system $(M,X)$ and element $x_0\in \upsilon (M)$, there is morphism $f:\big(\R,\frac{d}{dt}\big)\rightarrow (M,X)$ with $f(0) = x_0$ (existence).  In fact, there is only one (uniqueness) (\cref{theorem:E&U}).  This proves that $ \big((\R,\frac{d}{dt}),0\big)$ is initial in the category of elements, and therefore that $\upsilon:\sF{DySys}\rightarrow\sF{Set}$ is representable. 
\end{proof}
We conclude from \cref{prop:e&uRepresentable} that $\left(\R,\frac{d}{dt}\right)$ is a \textit{universal complete dynamical system} with time at zero---which defines initial condition---witnessing the natural isomorphism. There is nothing sacrosanct about zero time, and any other time for cataloging a prespecified state is equally suitable:  $\left(\left(\R,\frac{d}{dt}\right),t_0\right)$ is also initial in $\disg \upsilon$.  Of course, there is unique isomorphism $\left(\left(\R,\frac{d}{dt}\right),0\right) \cong \left(\left(\R,\frac{d}{dt}\right),t_0\right)$ by the translation map $t\mapsto t+t_0$.

\subsection{Local Time is a Universal Continuous-Time Dynamical System}\label{subsec:timeDS}
The Existence and Uniqueness  Theorem  (\cref{theorem:E&U}) does not require  solutions to be complete. If there is some analog of \cref{prop:e&uRepresentable} for possibly non-complete dynamical systems $\sF{DySys}$, we must exhibit finesse in defining the relevant categories and functors. For example, the ode $\dx = x^2$ has maximal solutions with different domains. 

 Noting that our definition of continuous-time dynamical system $(M,X)$ (\cref{def:c0-dysys}) consists of a state space (the manifold $M$) with some dynamics ($X$) on that space, we prepare our setting by isolating components and first augmenting the state space category. We then define dynamics as some section of an appropriate bundle over that state space. In \cref{sec:abstractSystem}, we will  abstract this observation for a general theory of dynamical systems. 

\begin{definition}\label{def:c-manifolds}
We define the category $\sF{Man}$ of smooth manifolds to have \begin{enumerate}
	\item smooth manifolds $M$ as objects, and 
	\item smooth maps $M\xrightarrow{f}N$ between manifolds as morphisms. 
\end{enumerate}
\end{definition}

By singling out the state space category, we remark that \cref{def:c-dysys} remains unchanged: a map of dynamical systems is a map in the category $\sF{Man}$ of manifolds preserving some structure on the dynamics. Now for the augmentation: 

\begin{definition}\label{def:germManifolds}
	We define the category $\sF{gMan}$ of germed Manifolds to have
	\begin{enumerate}
		\item smooth manifolds $M$ as objects, and 
		\item  equivalence classes $(U\subset M)\xrightarrow{[f_U]}N$ of partial (smooth) maps for open $U\subset M$ as morphisms. 	\end{enumerate}
	\end{definition} 
	Two partial maps $(U\subset M)\xrightarrow{f_U}N$ and $(V\subset M)\xrightarrow{f_V}N$ are \textit{equivalent}, denoted $f_U\sim f_V$, if they both contain a common basepoint $x_0\in U\cap V$ and their restrictions are equal $f_U|_{U\cap V} = f_V|_{U\cap V}$. This equivalence defines a \textit{germ} of functions and the basepoint is essential for the relation to be an equivalence.

Composition of partial maps for sets is defined as follows: the composition $g\circ f \subset Z\times X$ in $$\begin{tikzcd}X\arrow[r,"f"]\arrow[rr,bend left,"g\circ f"] & Y \arrow[r,"g"] & Z\end{tikzcd}$$ is the relation $g\circ f \defeq\big\{(z,x)\in  Z\times (dom(f)\subset X):\,z=g(y)\,\mbox{for some}\, y=im(f)\cap dom(g)\big\}$, with $dom(g\circ f) = dom(f) \cap f^{-1}(dom(g))$.  That sets with partial maps form a category requires a straightforward verification that composition is associative. The check for $\sF{gMan}$ is similar.

\begin{example}\label{ex:puncturedR}
Consider, for  example, the punctured real line $\big(\R\setminus\{t'\}\big)$ as an object of $\sF{gMan}$. There are two partial identity maps  $\begin{tikzcd}\R\arrow[r,shift left,"id_r"]\arrow[r,shift right,swap,"id_\ell"] & \big(\R\setminus\{t'\}\big),\end{tikzcd}$ mapping $t\mapsto t,$ with $dom(id_\ell) =(-\infty,t')$ and $dom(id_r)= (t',\infty)$. These two  maps define two distinct partial identity morphisms $[id_{r/\ell}]:\R\rightarrow\R\setminus\{t'\}$ in $\sF{gMan}$. 
\end{example}


We now consider germed dynamical systems $\sF{gDySys}$:
\begin{definition}\label{def:c-gdysys}
	We define the category $\sF{gDySys}$ of \textit{germed (continuous-time) dynamical systems} with the following data: 
	\begin{enumerate}
		\item objects $(M,X)$ are continuous-time dynamical systems (\cref{def:c-dysys}), and
		\item morphisms $(M,X)\xrightarrow{f}(N,Y)$ are maps $M\xrightarrow{[f]}N$ of germed manifolds (\cref{def:germManifolds}) satisfying the property that  $(U,X|_U)\xrightarrow{f_{U}}(N,Y)$ is a map of dynamical systems (\cref{def:c1-dysys}) for  representative $f_{U}$ of $[f]$.
	\end{enumerate}
\end{definition} 

The forgetful functor $\upsilon:\sF{gDySys\rightarrow Set}$ sending $(M,X)\mapsto \big\{x\in M\big\}$ forgets dynamics and manifold structure, and on  morphism $(M,X)\xrightarrow{[f]}(N,Y)$ is defined as $\upsilon (f) \defeq   \disu_{f_U\in [f]}dom(f_U)\xrightarrow{\bigcup f_{U}}N$.  That manifolds have sheaves of regular functions and vector fields   makes it easy to check that this definition is functorial.

The represented functor \begin{equation}\label{eq:repfun-gdysys}
\sF{gDySys}\left(\left(\R,\frac{d}{dt}\right),\bullet\right):\sF{gDySys\rightarrow Set}
\end{equation} takes a germed dynamical system $(M,X)$ to germs of integral curves of $(M,X)$. Integral curves extend to maximal curves (\cref{theorem:E&U}), and specification of initial condition $x_0\in M$ defines unique map $\ph_{X,x_0}\in  \sF{gDySys}\left(\left(\R,\frac{d}{dt}\right),(M,X)\right)$ which takes $0\mapsto  \ph_{X,x_0}(0)=x_0$. We  conclude that the map $\ph:\upsilon\Rightarrow \sF{gDySys}\left(\left(\R,\frac{d}{dt}\right),\bullet \right)$ assembles into a natural transformation

\begin{equation}\label{eq:naturalityOfgExp}
\begin{tikzcd}	
\big\{x\in M\big\}\arrow[rr,"\ph_{X,(\cdot)}"]\arrow[dd,"\upsilon (f)"] & & \sF{gDySys}\left(\left(\R,\frac{d}{dt}\right),(M,X)\right)\arrow[dd,"(f_*)"] \\ 
\\ 
\big\{y\in N\}\arrow[rr,"\ph_{Y,(\cdot)}"] & &  \sF{gDySys}\left(\left(\R,\frac{d}{dt}\right),(N,Y)\right), 
\end{tikzcd}
\end{equation}  and is in fact a natural \textit{isomorphism}  (c.f.\ \eqref{eq:naturalityOfExp}). The reasoning is similar to that of \cref{prop:e&uRepresentable} and proves the followin analog: 
\begin{prop}\label{prop:e&uGRepresentable}
	The forgetful functor $\upsilon:\sF{gDySys\rightarrow Set}$ is representable, witnessed by initial object $\left(\left(\R,\frac{d}{dt}\right),0\right)\in \disg_{\sF{gDySys}}\upsilon$ in the category of elements.
\end{prop}

\begin{example}\label{ex:puncturedRDS}
	 We consider the punctured line (\cref{ex:puncturedR}) as a  continuous-time dynamical system $(M,X)=\left(\R\setminus\{0\},\frac{d}{dt}\right)$. Choice of initial condition $t_0> 0$ (the case where $t_0<0$ is similar) defines unique maximal solution $(-t_0,\infty)\rightarrow (0,\infty)$ with $\ph_{X,t_0}(0)=t_0$. 
\end{example}

\subsection{Discrete Time is a Universal Discrete-Time Dynamical System}\label{subsec:timeDiscrete}
We include an example straight from \cite[Examples 2.1.1, 2.4.11]{riehl} for expositional completeness. At the present moment, we will be cursory on formalism and consider the matter with finer detail in \cref{ex:iteratedMapasAbstractSolution}.  

 A discrete-time dynamical system is an endomap $X:\sF{c\rightarrow c}$ of  set $\sF{c\in Set}$, which we denote as pair $(\sF{c},X)$; typically there is  specification of basepoint $c_0\in \sF{c}$ as well, which we will recover momentarily. A map $\a:(\sF{c},X)\rightarrow(\sF{c'},Y)$ of discrete-time systems is a map $\a:\sF{c\rightarrow c'}$ of sets  which respects discrete-time dynamics: $Y\circ \a = \a \circ X$. We denote the category by $\sF{dtDySys}$. Discrete time itself arises as a discrete-time system $(\N,s)$ with $s:\N\rightarrow \N$ mapping $n\mapsto n+1$. When basepoint is specified, the map should respect basepoint.   Inclusion of basepoint  similarly arises by representability of the appropriate forgetful functor $\upsilon:\sF{dtDySys\rightarrow Set}$ sending $(\sF{c},X)\mapsto \{c\in \sF{c}\}$. In this setting, initial object $\big((\N,s),0\big)\in \disg_{\sF{dtDySys}}\upsilon$ in the category of elements witnesses its representability, with universal dynamical system $(\N,s)$ and universal element $0\in \N$. In other words, for $\big((\fc,X),c_0\in \fc\big)$, there is unique map of discrete-time systems $\big((\Nb,s),0\big)\dashrightarrow\big((\fc,X),c_0\big)$, and $c_0$ is the basepoint of system $(\fc,X)$.

\section{Abstract Systems}\label{sec:abstractSystem}

 We now consider an abstract formalism of dynamical systems which unifies the examples in \cref{subsec:timeCDS}, \cref{subsec:timeDS}, and \cref{subsec:timeDiscrete}. Our motivation is to identify a notion which both captures the relevant concept of time as well as provide guidance for how to rigorously construct other classes of dynamical systems.

\begin{definition}\label{def:CfiberedInDandSections}\label{def:fiberedTransformationInD}
Let $\fC, \fD$ be locally small concrete categories,  $\Tc,\Uc:\fC\rightarrow\fD$   functors, and natural transformation  $\nat{\fC}{\fD}{\Tc}{\Uc}{\t}$. We say that $\t$ is $\fD$-\textit{fibered} (or simply \textit{fibered}) if  $\Tc\fc\xrightarrow{\t_\fc}\Uc \fc$ is a split epimorphism for each object $\fc\in \fC_0$. 
For a $\fD$-fibered transformation, we define  $\t$-\textit{sections} by the set of right inverses of $\t_\fc$: 
 \begin{equation}\label{eq:abstractSections} \G_\t(\fc)\defeq \left\{\big(\Uc\fc\xrightarrow{X} \Tc\fc\big)\in \fD_1:\, \t_{\fc}\circ X =id_{\Uc\fc}\right\}.\end{equation}
	 \end{definition}
	 
	 \begin{remark}\label{remark:whySplit?}
	 	Since $\t_\fc:\Tc\fc\rightarrow\Uc\fc$ is split epi, $\t$-sections $\G_\t(\fc)$ are guaranteed to be nonempty. 
	 \end{remark}
	
This definition provides a way to package ingredients used in various notions of systems.   

\begin{example}\label{ex:identityFunctorForMan}
Let $\fC=\fD=\sF{Man}$, $\Uc = id_{\sF{Man}}$, and $\Tc:\sF{Man}\rightarrow\sF{Man}$ assign the tangent bundle $TM$ to each manifold $M$.  This assignment is functorial (chain rule). Moreover, the canonical projection of the tangent bundle $\t_M:TM\rightarrow M$ assembles into a natural transformation.  Finally, the  projection $\t_M:TM\rightarrow M$ is a split epimorphism, whose sections $\G_\t(M)=\Xf(M)$ are smooth vector fields in $M$.  Therefore the natural transformation $\t$ is $\sF{Man}$-fibered, or in this case, simply fibered. 
\end{example}

\begin{example}\label{ex:identityFunctorForGermedMan}
Let $\fC=\fD=\sF{gMan}$, $\Uc = id_{\sF{gMan}}$, and $\Tc:\sF{gMan}\rightarrow\sF{gMan}$ assign the germed tangent bundle $TM$ to germed manifold $M$.  This example is nearly identical to \cref{ex:identityFunctorForMan}. 
\end{example}

\begin{example}\label{ex:discSysAsprojCsysInC}
We consider the case of discrete time systems from \cref{subsec:timeDiscrete}. 	Let $\fC$ be a concrete category with products,  $\fD = \fC$, and let  $\Tc = (\cdot)\times (\cdot),$  $\Uc = id_\fC$.\footnote{Formally, $\Tc = (\cdot)\times (\cdot)$ is the composition of functors $\fC\xrightarrow{\Delta}\fC\times \fC\xrightarrow{\times}\fC$ sending $\fc\mapsto (\fc,\fc)\mapsto \fc\times \fc$.}  For object $\fc \in \fC$, we have projection $\fc\times\fc \xrightarrow{p_i}\fc$ onto the $i$th factor and set $\t = p_1$ to be the projection onto the first factor.  A section $X\in \G_\t(\fc)$ is a map $\fc\xrightarrow{X}\fc\times\fc$  for which $X(x)=(X_1(x),X_2(x)) = (x,X_2(x))$ acts as identity on the first factor. \end{example}

In \cref{ex:identityFunctorForMan}, \cref{ex:identityFunctorForGermedMan},  and \cref{ex:discSysAsprojCsysInC} the role of $\fC$ appears superfluous. The formalism may be used to define ``dynamics in $\fD$'' for kinds of ``systems in $\fC$.'' 

\begin{example}
	Hybrid systems, ubiquitous in engineering, are studied in \cite{goebelhybrd}, \cite{ames}, \cite{savkin}. The categorical formulations  in \cite{schmidt2019morphisms} and \cite{lermanSchmidt1} align well with the formalism presented here. A hybrid phase space $a:\Sb^a\rightarrow\sF{RelMan}$ is a functor from discrete double category to double category $\sF{RelMan}$ (whose morphisms are maps of manifolds and relations). Hybrid phase spaces form a category $\sF{HyPh}$ (\cite[Lemma 3.2.6]{schmidt2019morphisms}) and there is  forgetful functor $\Ub:\sF{HyPh\rightarrow Man}$ taking $a\mapsto \discats_{s\in \Sb^a_0}a(s)$.  There is also tangent functor $\Tc\defeq T\circ \Ub$ for tangent endofunctor $T:\sF{Man\rightarrow Man}$, and the canonical projection $\t:\Tc\Rightarrow \Ub$ defines a $\sF{Man}$-fibered natural transformation, which is split epi.
\end{example}

The formalism of fibered transformation in \cref{def:CfiberedInDandSections} organizes ingredients we need for a construction of abstract system. With these pieces in place, we may now define:

\begin{definition}\label{def:DSystemsFromC}\label{def:tauSystem}\label{def:abstractSystems}
	Let $\Tc,\Uc:\fC\rightarrow \fD$ be functors, and $\nat{\fC}{\fD}{\Tc}{\Uc}{\t}$ a $\fD$-fibered transformation (\cref{def:fiberedTransformationInD}). We define $\t$-\textit{system} $(\fc,X)$ to be a pair  where $\fc\in \fC_0$ is an object in $\fC$ and $X\in \G_\t(\fc)$ is a  $\t$-section. We also define   \textit{morphism $(\fc,X)\xrightarrow{f}(\fd,Y)$ of $\t$-systems} to be a morphism $\fc\xrightarrow{f}\fd$ in $\fC$ such that $(X,Y)$ are $f$-related, i.e.\ $\Tc f \circ X = Y\circ \Uc f$.  The collection of $\t$-systems and morphisms make up a category, which we denote by $\t$-$\sF{Sys}$.
\end{definition}

\begin{example}\label{ex:dySysAsManSysInMan}
	Let $\fC=\fD=\sF{Man}$, $\Uc = id_{\sF{Man}}$, $\Tc=T$ the tangent functor (\cref{ex:identityFunctorForMan}), and $\t:\Tc\Rightarrow\Uc$ the canonical projection of the tangent bundle, mapping $\big(v_p\in T_pM\big)\mapsto \big(p\in M\big)$.  Then a $\t$-system $(M,X)$  is a continuous-time dynamical system (\cref{subsec:timeCDS}). 
\end{example}
   
\begin{example}\label{ex:gDySysAsAbstractSystem}
Similarly for \cref{ex:identityFunctorForGermedMan}, let $\t:\Tc\Rightarrow\Uc$ be the canonical projection of germed trangent bundle. A $\t$-system $(M,X)$ is then a germed (continuous-time) dynamical system (\cref{subsec:timeDS}). 
\end{example}

	\begin{example}\label{ex:dtDySysAsAbstractSystem}
		For $\t:\Tc\Rightarrow\Uc$ as in \cref{ex:discSysAsprojCsysInC}, a $\t$-system $(\fc,X)$ is a discrete-time dynamical system (\cref{subsec:timeDiscrete}). 
	\end{example}

Still missing from the theory of systems in our abstraction is a notion of solution, which we saw in various forms in \cref{sec:dynamicalSystems}.  Taking a hint from \cref{prop:e&uRepresentable} and \cref{prop:e&uGRepresentable},  we \textit{define} a solution for abstract systems as a morphism from an initial object in the relevant category of elements.

\begin{definition}\label{def:abstractSolution}
	Let $(\fc,X)\in \t\text{-}\sF{Sys}$ be an abstract $\t$-system, and let $\upsilon:\sF{C\rightarrow Set}$ be the forgetful functor.\footnote{We assume in \cref{def:tauSystem} that $\sF{C}$ is concrete and therefore a subcategory of $\sF{Set}$.} 	We say that $\t$-systems \textit{have solutions} when there is an initial object $((\sF{t},T),t_0)\in\disg_{\t\text{-}\sF{Sys}} \upsilon$ in the category of elements and that a map $(\sF{t},T)\rightarrow(\fc,X)$ is a \textit{solution of} $(\fc,x)$. 
The system $(\sF{t},T)$ is a \textit{universal} $\t$-system, which we also call $\t$-time,  and $t_0\in \upsilon(\sF{t})$ is a \textit{universal element} witnessing universality of $\t$-time. 
\end{definition}

 When $\upsilon:\sF{C\rightarrow Set}$ is representable in $\t$-$\sF{Sys}$, solutions exist.  One may take representability to be a constitutive criterion of abstract $\t$-systems. We identify a few examples.

\begin{example}\label{ex:iteratedMapasAbstractSolution} 
We continue example \cref{ex:dtDySysAsAbstractSystem} translating example  \cite[2.4.11]{riehl}.
 Let $X:\fc \rightarrow \fc $ be a discrete dynamical system (\cref{ex:discSysAsprojCsysInC}) and suppose that the natural numbers $\N\in \fC$.  Consider successor map $\s:\N\rightarrow\N$ defined by $\s(n)\defeq n+1$. This map defines a discrete dynamical system, and a map $(\N,\s)\xrightarrow{\a}(\fc,X)$ of discrete-time systems satisfies \begin{equation}\label{eq:discRelated} \a\circ\s = X\circ \a.\end{equation}  Choosing initial point $c_0$ as the image of $0$ under $\a$, we have entirely determined the map $\a$: for $\a$-relatedness in \eqref{eq:discRelated} implies that $\a(1) = a(\s(1)) = X(c_0)$ and in general, $\a(n) = \underbrace{X\circ \cdots \circ X}_{\text{n-times}}(c_0).$ 
 Thus the map $$\left(\big(\N,\s\big),0\right)\xrightarrow{\a}\left(\big(\fc,X\big),c_0\right)$$  in $\disg_{\sF{dSys}}\upsilon$ defines a solution of $(\fc,X)$ in the sense of \cref{def:abstractSolution}.
\end{example}

\begin{example}\label{ex:intCurveasAbstractSolution}
	In \cref{ex:gDySysAsAbstractSystem}, we identified germed dynamical systems as instance of abstract $\t$-system and in \cref{prop:e&uGRepresentable} saw that $\upsilon:\sF{gDySy\rightarrow Set}$ is representable, and therefore has solutions in the sense of \cref{def:abstractSolution}. 
\end{example}

Examples witnessing the generalization of systems from \cref{def:abstractSystems} appear in \cite{schmidt2019morphisms}. Conditions guaranteeing representability of functor $\upsilon$ for arbitrary $\t$-systems may support development of a general theory of systems with this formalism guiding concrete constructions. Of course, 
uniqueness of solutions in \cref{def:abstractSolution} assumes determinism. Therefore, time in non-deterministic systems may not exhibit the same universal property. This theory seems immediately suitable for deterministic systems, for which  uniqueness of behavior is key.

%% file: CategoryDynamicalSystem.bbl
\begin{thebibliography}{1}

\bibitem{ames}
{\sc A.~Ames}, {\em A Categorical Theory of Hybrid Systems}, PhD thesis,
  California Institute of Technology, 2006.

\bibitem{goebelhybrd}
{\sc R.~Goebel, R.~Sanfelice, and A.~Teel}, {\em Hybrid Dynamical Systems:
  Modeling, Stability, and Robustness}, Princeton University Press, 2012.

\bibitem{lermanhybrid}
{\sc E.~Lerman}, {\em A category of hybrid systems}, arXiv:1612.01950,  (2016).

\bibitem{lermanSchmidt1}
{\sc E.~Lerman and J.~Schmidt}, {\em Networks of hybrid open systems}, Journal
  of Geometry and Physics, 149 (2020), p.~103582.

\bibitem{riehl}
{\sc E.~Riehl}, {\em Category Theory in Context}, Dover, 2016.

\bibitem{savkin}
{\sc A.~V. Savkin and R.~J. Evans}, {\em Hybrid Dynamical Systems Controller
  and Sensor Switching Problems / by Andrey V. Savkin, Robin J. Evans.},
  Control Engineering, Birkhäuser Boston, Boston, MA, 1st ed. 2002.~ed., 2002.

\bibitem{schmidt2019morphisms}
{\sc J.~Schmidt}, {\em Morphisms of networks of hybrid open systems}, arXiv
  [math.DS] \href{https://arxiv.org/abs/1911.09048}{arXiv:1911.09048},  (2019).

\bibitem{tu}
{\sc L.~Tu}, {\em An Introduction to Manifolds}, Springer, 2011.

\bibitem{willems}
{\sc J.~C. Willems}, {\em The behavioral approach to open and interconnected
  systems}, IEEE Control Systems Magazine, 27 (2007), pp.~46--99.

\end{thebibliography}
